\documentclass[11pt,reqno]{amsart}

\usepackage[margin=1.1in]{geometry}
\usepackage[T1]{fontenc}
\usepackage{lmodern}
\usepackage{amsmath,amssymb,amsthm,mathtools}
\usepackage{microtype}
\usepackage{xcolor}
\usepackage{tikz}
\usepackage{flafter}
\definecolor{linkblue}{RGB}{25,55,115}
\usepackage[colorlinks=true,linkcolor=linkblue,citecolor=linkblue,
  urlcolor=linkblue,bookmarksnumbered=true]{hyperref}
\hypersetup{pdftitle={A law of large numbers for Macdonald coherent measures}}

\newtheorem{theorem}{Theorem}[section]
\newtheorem{proposition}[theorem]{Proposition}
\newtheorem{lemma}[theorem]{Lemma}
\newtheorem{corollary}[theorem]{Corollary}
\theoremstyle{definition}

\theoremstyle{remark}
\newtheorem{remark}[theorem]{Remark}
\numberwithin{equation}{section}

\newcommand{\Y}{\mathbb Y}
\newcommand{\R}{\mathbb R}
\newcommand{\Prob}{\mathbb P}
\newcommand{\E}{\mathbb E}
\newcommand{\ddim}{\operatorname{dim}_{q,t}}
\newcommand{\tdim}{\operatorname{dim}_{t,q}}
\newcommand{\wt}{\operatorname{wt}}
\newcommand{\eps}{\varepsilon}
\DeclareMathOperator{\Bin}{Bin}

\begin{document}

\title{A law of large numbers for Macdonald coherent measures}

\author{Vadim Gorin}

\thanks{The work was partially supported by NSF Grant DMS-2553163.}

\address{University of California at Berkeley}
\email{vadicgor@gmail.com}

\date{}

\begin{abstract}
We resolve a conjecture on the law of large numbers for
coherent measures on Young diagrams associated with Macdonald-nonnegative specializations of the algebra of symmetric functions for all $0\le q,t<1$. We prove that row and column lengths divided by the diagram size
converge to the corresponding specialization parameters. The proof combines a
monomial estimate for skew Macdonald polynomials with exponential tilting
and binomial decompositions obtained from the branching rule for the polynomials.
\end{abstract}

\maketitle

\section{Introduction}\label{sec:intro}

\subsection{History of the problem}
Thoma parameters  $\alpha_1\ge \alpha_2\ge\dots\ge 0$, $\beta_1\ge \beta_2\ge\dots\ge 0$, ${\sum_{i}(\alpha_i+\beta_i)\le 1}$, \cite{Edr52,Tho64}, classify extreme characters of
the infinite symmetric group on the one hand and govern the macroscopic geometry of
random Young diagrams on the other hand. The two viewpoints are linked by the Vershik-Kerov law of
large numbers \cite{VK81}, which identifies characters with
coherent systems of probability measures on Young diagrams (relying on the branching rule for irreducible representations of finite symmetric groups) and recovers the parameters as the limiting row
and column frequencies for growing diagrams. Various approaches to the Thoma classification theorem can be found in \cite{Tho64,VK81,KOO98,Oko99,BG15}, see also the books \cite{Ker03,BO16} for overviews.

The classical Frobenius characteristic map can be used to identify extreme characters with
normalized Schur-nonnegative specializations of the algebra of symmetric functions. This reformulation leads to a natural $(q,t)$ deformation of the problem by replacing Schur functions with Macdonald symmetric functions (see \cite{Mac95}). Schur functions are recovered at $q=t$, while other important cases include Jack functions ($t=q^{\theta}$, $q\to 1$) related to representations of infinite-dimensional Gelfand pairs, Hall-Littlewood functions ($q=0$) related to matrices over finite fields, monomial functions ($t=1$) related to Kingman partition structures. Macdonald-nonnegative
specializations for the general $(q,t)$ case play an important role in the probabilistic framework of Macdonald processes, see \cite{BC14,BCGS16}.

Kerov conjectured in 1992 \cite[Section~7.3]{Ker92} that the classification of all Macdonald-nonnegative specializations is governed by the same Thoma parameters, yet the formulas for the specializations depend on $(q,t)$ in a non-trivial fashion. After 25 years of effort, this was proven in a breakthrough paper by Matveev \cite{Mat19}. The corresponding interpretation of the parameters as row and column frequencies was also widely expected to extend to the general $(q,t)$ case, as reflected in the literature from the Jack case \cite{KOO98}, the 1997-2000 drafts of Vershik and Kerov on matrices over finite fields \cite{VK07}, the $q=0$ result of \cite{BP15}, and the explicit formulation of the general $(q,t)$ case in \cite[Conjecture~1.9]{BP15}. However, the proof of \cite{Mat19} does not deliver this interpretation, while the methods of \cite{KOO98} and \cite{BP15} do not extend to general $(q,t)$.

The goal of this paper is to prove the probabilistic counterpart of Kerov's conjecture and identify the parameters of specializations with frequencies of rows and columns for all $0\le q,t<1$.

\subsection{The measures and the main result}
Let $\Y_n$ be the set of partitions of $n$, identified with Young diagrams.
For $\lambda\in\Y_n$, write $\lambda_i$ and $\lambda'_j$ for its row and
column lengths, respectively. Fix
\[
  0\le q,t<1,\qquad\text{ and set }\qquad \theta=\frac{1-t}{1-q}.
\]
Let $\Lambda$ be the algebra of symmetric functions in infinitely many variables $(x_1,x_2,\dots)$,
and let $P_\lambda(\cdot;q,t)\in\Lambda$ be the Macdonald symmetric functions, as in \cite[Chapter~VI]{Mac95}. Define the positive numbers
$\ddim(\lambda)$ using $p_1=x_1+x_2+x_3+\dots$ as the coefficients in the decomposition
\begin{equation}\label{eq_dimension_expansion}
  p_1^n=\sum_{\lambda\in\Y_n}\ddim(\lambda)P_\lambda.
\end{equation}
When $q=t$, these are dimensions of irreducible representations of the symmetric group $S_n$ indexed by $\lambda$. A \emph{specialization} is a unital algebra homomorphism $\rho:\Lambda\to\R$. We write $f(\rho)$ for
the image of a symmetric function $f$. The specialization is \emph{Macdonald-nonnegative}
if $P_\lambda(\rho)\ge0$ for every $\lambda$.
We call it normalized if $p_1(\rho)=1$. By \cite[Theorem 1.4]{Mat19} every such
specialization is uniquely determined by parameters
\begin{equation}\label{eq_parameter_space}
 \alpha_1\ge\alpha_2\ge\cdots\ge0,\qquad
 \beta_1\ge\beta_2\ge\cdots\ge0,\qquad \gamma\ge0,\qquad
 \sum_{i=1}^{\infty}\alpha_i+\sum_{j=1}^{\infty}\beta_j+\gamma=1.
\end{equation}
The corresponding specialization is defined on power sums $p_r=\sum_i (x_i)^r$ which generate $\Lambda$:
\begin{equation}\label{eq_power_sums_specialization}
 p_1(\rho)=1,\qquad
 p_r(\rho)=\sum_{i=1}^{\infty}\alpha_i^r+
 (-1)^{r-1}\theta^r\frac{1-q^r}{1-t^r}\sum_{j=1}^{\infty}\beta_j^r,
 \quad r\ge2.
\end{equation}
The relation to the conventions in
\cite{BC14,BP15,Mat19} is given in Section~\ref{sec:specializations}. The associated probability measure on $\Y_n$ is
\begin{equation}\label{eq_coherent_measure}
 M_n^\rho(\lambda)=\ddim(\lambda)P_\lambda(\rho),\qquad \lambda\in\Y_n.
\end{equation}
The normalization $\sum_{\lambda\in\Y_n} M_n^\rho(\lambda)=1$ follows by applying $\rho$ to
\eqref{eq_dimension_expansion}.

\begin{theorem} \label{Theorem_uniform}
 Let $0\le q,t<1$, and let $\rho$ be the normalized Macdonald-nonnegative
specialization with parameters \eqref{eq_parameter_space}. For every $\varepsilon>0$
there exist $C,c>0$ such that for all $n\ge 1$
\begin{equation}\label{eq_uniform_concentration}
 M_n^\rho\left[
 \max\left(\sup_{i\ge1}\left|\frac{\lambda_i}{n}-\alpha_i\right|,
          \, \sup_{j\ge1}\left|\frac{\lambda'_j}{n}-\beta_j\right|\right)
 >\varepsilon\right]\le Ce^{-cn}.
\end{equation}
\end{theorem}

There is a natural way to couple the measures $M_n^\rho$ over all $n=1,2,\dots$, see Section \ref{Section_coherence},
arriving at a probability measure on growing Young diagrams (or random paths in the Young graph). This is a consequence of the \emph{coherency relationship}\footnote{Hence, the name ``coherent measures'' for $M_n^\rho$.} between $M_n^{\rho}$ and $M_{n-1}^\rho$, which says that the latter can be obtained from the former by one step of a canonical
Markov chain not depending on $\rho$.
Under this coupling of the measures $M_n^{\rho}$, \eqref{eq_uniform_concentration} and the Borel-Cantelli lemma imply almost sure convergence:
$$
 \lim_{n\to\infty} \frac{\lambda_i}{n}=\alpha_i,\qquad \lim_{n\to\infty} \frac{\lambda'_j}{n}=\beta_j, \qquad i,j\ge 1.
$$

\begin{remark}
The restriction $t<1$ is essential in Theorem \ref{Theorem_uniform}. At $t=1$, with $0\le q<1$,
the Macdonald polynomials reduce to monomial symmetric functions,
and the boundary becomes the Kingman simplex, parametrized by
$\alpha_1\ge\alpha_2\ge\cdots\ge0$ with $\sum_i\alpha_i\le1$. There are no dual parameters $\beta_j$ anymore, cf.\
\cite[Section~4]{KOO98}. The column conclusion of \eqref{eq_uniform_concentration} then fails:
the specialization $p_1=1$, $p_r=0$ for $r\ge2$, produces
the deterministic diagram $(1^n)$, with $\lambda_1'=n$.
\end{remark}

\subsection{Related work} For special values of $(q,t)$, Theorem \ref{Theorem_uniform} connects to extensive literature.

\subsubsection{Schur measures and random words}
When $q=t$, the measures \eqref{eq_coherent_measure} arise from independent
letters through the generalized Robinson--Schensted--Knuth correspondence
of \cite{KV86}. The $\alpha_i$ and $\beta_j$ parameters give
the probabilities of two types of discrete letters, with opposite
repetition conventions. \cite[Theorem~2]{Buf12} used this construction to get both
a law of large numbers and a central limit theorem for $\lambda_i$ and $\lambda'_j$.
In parallel, \cite{Mel12} obtained a central limit theorem through
the algebra of observables of Young diagrams. Further developments in \cite{FMN20} give concentration inequalities for polynomial observables, which also imply the $q=t$ case of Theorem \ref{Theorem_uniform}. For finitely many $\alpha_i$ and no other non-zero parameters, the law of large numbers
and finer fluctuation results also appear in Its, Tracy, and Widom
\cite[Section~3]{ITW01} and Houdr\'e and Xu \cite{HX13}.  Poissonizing \eqref{eq_coherent_measure} at $q=t$ produces a Schur measure of \cite{Oko01}, whose particle configuration $\{\lambda_i-i\}_{i\ge 1}$ is
a determinantal point process. Through steepest descent analysis of the contour integral
representation of the kernel and subsequent de-Poissonization  (cf.\ \cite{BOO00,Joh01}), it should
again be possible to recover the $q=t$ case of Theorem \ref{Theorem_uniform}. The determinantal structure is not expected to extend to the general $(q,t)$ setting analyzed in our work.

\subsubsection{Hall--Littlewood measures and matrices over finite fields}
At $q=0$ and $t=\mathfrak q^{-1}$, with $\mathfrak q$ a prime power,
the measures $M_n^\rho$ describe Jordan types of finite corners of
infinite unipotent upper-triangular matrices over $\mathbb F_{\mathfrak q}$.
The associated law of large numbers was conjectured by Vershik and Kerov
in drafts written in 1997--2000 \cite[Part~II, Section~4]{VK07}; see also
\cite[Conjecture~4.5]{GKV14}. \cite{Bor99} proved the law
of large numbers and a central limit theorem for the uniform triangular
matrix model. \cite{BP15} proved the Hall--Littlewood law of large numbers for $\gamma=0$ using randomized RSK. Hence, a corollary of Theorem \ref{Theorem_uniform}
is the extension of the law of large numbers of Borodin, Bufetov, and Petrov to all parameters \eqref{eq_parameter_space}, completing the
proof of the Vershik-Kerov law-of-large-numbers conjectures at $q=0$.

\subsubsection{Macdonald processes}
The measures \eqref{eq_coherent_measure} are fixed-partition-size versions of Macdonald
measures investigated in \cite{BC14,BCGS16}. These articles used Macdonald
difference operators and their variations to produce expressions for expectations of symmetric polynomials in variables
$\{q^{\lambda_i}t^{1-i}\}$. Such observables yield extensive asymptotic information, including various laws of large numbers for $q=t$, $q=0$, $t=0$, and $q,t\to 1$ regimes, see, e.g., \cite{BC14,BoG15,Ahn20}. However,
these methods have not yielded the law of large numbers considered here for general fixed $0<q,t<1$. A technical reason is that for $0<q<1$ and growing $\lambda_i$, the expectation $\E q^{\lambda_i}$ can be dominated not by typical behavior of $\lambda_i$ (such as the law of large numbers), but rather by large-deviation events.

\subsubsection{Geometric Martin boundaries for branching graphs} The law of large numbers can follow for free from general boundary theory in  situations when the Martin boundary of a branching graph (of which the Young graph governing branching of irreducible representations of symmetric groups is an example) can be identified. This is the case for the Jack-deformed Young graph ($q,t\to 1$ limit of our setting) in \cite{KOO98}, for the graph of projective representations of symmetric groups ($q=0$, $t=-1$) in \cite{Naz92,Iva99}, for the Gelfand-Tsetlin graph (branching graph for unitary groups) in \cite{OO98}, its BC-type deformation in \cite{OO06}, $q$--deformations in \cite{Gor12,Pet14,GP15,CG20}, and a $(q,t)$ deformation related to Macdonald polynomials in \cite{Cue18,Ols21}. At a technical level, these papers studied asymptotics of symmetric polynomials as the number of variables goes to infinity by using expansions in interpolation polynomials or contour integral representations. While these approaches have been known since the 1990s, attempts to adapt them to the setting of Theorem \ref{Theorem_uniform} did not succeed.

\subsection{The proof} Our proof of Theorem \ref{Theorem_uniform} does not use any of the machinery
of the previous section. Instead, it is based on three ideas.

The first ingredient is an approximation for a skew Macdonald polynomial by a single monomial. For a pair of diagrams $\mu\subseteq\lambda$
with at most $d$ boxes in each column of the skew diagram $\lambda/\mu$, fill each column of this diagram from top to
bottom with $1,2,\ldots$, and let $\kappa_j$ be the number of entries
equal to $j$. We prove in Theorem \ref{Theorem_monomial_localization} that for $x_1\ge\cdots\ge x_d>0$,
\begin{equation}\label{eq_intro_localization}
 \log P_{\lambda/\mu}(x_1,\ldots,x_d)
 = \sum_{j=1}^d\kappa_j\log x_j + O_{q,t,d}\bigl(\mu_1+\log(1+\lambda_1)\bigr).
\end{equation}

Second, the identity \eqref{eq_dimension_expansion} implies that for $\rho=(\alpha_1,\dots,\alpha_d; 0; 0)$
\begin{equation}
\label{eq_expectation_identity}
 \E_{M^{\rho}_n} \left[\frac{P_\lambda(x_1,\dots,x_d)}{P_\lambda(\rho)} \right]= (x_1+\dots+x_d)^n.
\end{equation}
Combining with \eqref{eq_intro_localization} for $\mu=\emptyset$ and Markov's inequality, this leads to \eqref{eq_uniform_concentration} for this particular $\rho$.

Third, the full statement of Theorem \ref{Theorem_uniform} involves infinitely many $\alpha_i$, as well as dual parameters, and $\gamma$. In order to handle that, we use a coupling between measures with different parameters \eqref{eq_parameter_space}, which is a relative of Kerov's mixing construction, iteratively building general parameters \eqref{eq_parameter_space} as unions of specializations in single $\alpha$, single $\beta$, and single $\gamma$. Our version of this construction, given in Proposition \ref{Proposition_binomial}, goes in the opposite direction: given $M_n^\rho$-distributed $\lambda$, we construct a random subdiagram $\mu\subseteq\lambda$, which corresponds to a subset of parameters in $\rho$. The number of boxes in $\mu$ is random, but easily controlled: it is given by a binomial distribution. Due to the nature of this coupling, using only $\mu=\emptyset$ in \eqref{eq_intro_localization} as in \eqref{eq_expectation_identity} is no longer sufficient, and we need to use the full statement of \eqref{eq_intro_localization}. We also need to control the rows for specializations $\rho$ with no $\alpha$'s by a separate argument.

\bigskip

Our arguments also provide new, shorter proofs for the previously known particular cases $q=t$ and $q=0$. It would be interesting to see whether these ideas can be adapted to other related branching graphs studied in the recent literature, such as \cite{GK26} and \cite{CO26}.

\subsection{Organization of the paper}
Section \ref{sec:preliminaries} records the algebraic conventions. Section \ref{Section_binomial_coupling} explains how the coupling of two specializations works and establishes a skew version of \eqref{eq_expectation_identity}. Section \ref{Section_noalpha} proves Theorem \ref{Theorem_uniform} for rows in the special case $\alpha_1=\alpha_2=\dots=0$.
Section \ref{Section_localization_to_monomial} proves the monomial estimate \eqref{eq_intro_localization} for skew Macdonald polynomials.
Section \ref{Section_concentration} combines the ingredients of the previous three sections to complete the proof of Theorem \ref{Theorem_uniform}.

\section{Macdonald symmetric functions and coherent measures on Young diagrams}\label{sec:preliminaries}

We use \cite[Chapter~VI]{Mac95} as the basic reference for Macdonald
symmetric functions; many of the formulas are also recorded in \cite[Chapter 2]{BC14}. In this section we recall various known facts and notations which will be used later.

\subsection{Notations}\label{sec:identities}
All the symmetric functions of our interest are elements of the algebra of symmetric functions $\Lambda$ in infinitely many variables $(x_1,x_2,\dots)$, with real coefficients. One possible set of algebraic generators of $\Lambda$ is given by the power sums $p_r=\sum_i (x_i)^r$, $r=1,2,\dots$.  Unless parameters
are displayed, the Macdonald symmetric functions $P_\lambda\in\Lambda$, $Q_\lambda\in\Lambda$, and their skew analogues have
parameters $(q,t)$.

We write $\mu\subseteq\lambda$ for inclusion of Young diagrams and
$\mu\nearrow\lambda$ when $\mu\subset \lambda$ and the skew diagram $\lambda/\mu$ (i.e., set-theoretic difference of the sets of boxes of $\lambda$ and $\mu$) consists of one box. $|\lambda|=\lambda_1+\lambda_2+\dots$ is the total number of boxes in $\lambda$.
The length $\ell(\lambda)$ is the number of nonzero parts in $\lambda$. We write
\begin{equation}\label{eq:partial-sums}
 S_k(\lambda)=\sum_{i=1}^k\lambda_i,\qquad k\ge1,\qquad S_0(\lambda)=0.
\end{equation}
A horizontal strip is a skew diagram $\lambda/\mu$ which has at most one box in each column; a vertical
strip is a skew diagram $\lambda/\mu$ with at most one box in each row. For a box $s=(i,j)\in\lambda$, its
arm and leg lengths are
\[
 a_\lambda(s)=\lambda_i-j,\qquad l_\lambda(s)=\lambda'_j-i.
\]
Set
\begin{equation}\label{eq:b}
 b_\lambda(s)=
 \frac{1-q^{a_\lambda(s)}t^{l_\lambda(s)+1}}
      {1-q^{a_\lambda(s)+1}t^{l_\lambda(s)}},
 \qquad b_\lambda=\prod_{s\in\lambda}b_\lambda(s),
 \qquad Q_\lambda=b_\lambda P_\lambda,
\end{equation}
and set $b_\lambda(s)=1$ for $s\notin\lambda$.
The bases $P_\lambda$ and $Q_\lambda$ are dual with respect to the scalar product on $\Lambda$ defined on products of power sums $p_\lambda=\prod_{i\ge 1} p_{\lambda_i}$ via:
\begin{equation}\label{eq:scalar-product}
 \langle p_\lambda,p_\mu\rangle_{q,t}
 =\mathbf1_{\{\lambda=\mu\}}z_\lambda
   \prod_{i=1}^{\ell(\lambda)}\frac{1-q^{\lambda_i}}{1-t^{\lambda_i}},
\end{equation}
where $z_\lambda=\prod_{r\ge1}r^{m_r}m_r!$ and $m_r$ is the multiplicity
of $r$ in $\lambda$. Our convention gives
$\langle p_1,p_1\rangle_{q,t}=\theta^{-1}$ and agrees with that of \cite{KOO98} in the Jack limit $(q,t)\to 1$. The duality of the $P$- and $Q$-bases gives
\begin{equation}\label{eq:dimension-pairing}
 \ddim(\lambda)=\langle p_1^{|\lambda|},Q_\lambda\rangle_{q,t}.
\end{equation}

The one-box Pieri rule for Macdonald symmetric functions reads
\begin{equation}\label{eq:pieri}
 p_1P_\mu=\sum_{\lambda:\,\mu\nearrow\lambda}
 w_{q,t}(\mu,\lambda)P_\lambda,
\end{equation}
where for  $\lambda/\mu=\square=(r,c)$ the coefficient is
\begin{equation}\label{eq:edge-formula}
 w_{q,t}\bigl(\mu,\mu+(r,c)\bigr)
 =\prod_{i=1}^{r-1}G(\mu_i-c,r-1-i),\qquad 
 G(a,l)=\frac{(1-q^at^{l+2})(1-q^{a+1}t^l)}
                  {(1-q^at^{l+1})(1-q^{a+1}t^{l+1})}.
\end{equation}
All these coefficients are strictly positive for $0\le q,t<1$. Wherever it does not lead to confusion, we write $w=w_{q,t}$. The skew Macdonald symmetric functions are characterized by the branching identity
\begin{equation}\label{eq:branching}
 P_\lambda(X,Y)=\sum_{\mu\subseteq\lambda}
 P_\mu(X)P_{\lambda/\mu}(Y),
\end{equation}
where $X$ and $Y$ are two alphabets of variables and $(X,Y)$ is their union.
The skew functions satisfy the relation
\begin{equation}\label{eq:transitivity}
 P_{\lambda/\mu}(X,Y)=
 \sum_{\mu\subseteq\nu\subseteq\lambda}
 P_{\nu/\mu}(X)P_{\lambda/\nu}(Y).
\end{equation}
For a finite alphabet, the combinatorial formula reads
\begin{equation}\label{eq:tableau}
 P_{\lambda/\mu}(x_1,\ldots,x_d)
 =\sum_T\psi_T x^{\wt(T)},
 \qquad
 \psi_T=\prod_{j=1}^d\psi_{\lambda^{(j)}/\lambda^{(j-1)}},
\end{equation}
where the sum is over semistandard Young tableaux of shape $\lambda/\mu$ with entries
in $\{1,\ldots,d\}$: these are fillings of boxes of $\lambda/\mu$ in which rows are weakly increasing and columns are strictly
increasing. The vector $\wt(T)$ records the number of occurrences of each
entry, and
\[
 \mu=\lambda^{(0)}\subseteq\lambda^{(1)}\subseteq\cdots
 \subseteq\lambda^{(d)}=\lambda,
\]
where $\lambda^{(m)}$ is the union of $\mu$ and all the boxes of $\lambda/\mu$ with entries $\{1,2,\dots,m\}$. 
Each skew diagram $\lambda^{(m)}/\lambda^{(m-1)}$ is a
horizontal strip, and the corresponding weight $\psi$ is computed as 
\begin{equation}\label{eq:strip-weight}
 \psi_{\xi/\zeta}
 =\prod_{s\in R_{\xi/\zeta}\setminus C_{\xi/\zeta}}
       \frac{b_\zeta(s)}{b_\xi(s)},
\end{equation}
where $R_{\xi/\zeta}$ and $C_{\xi/\zeta}$ are the sets of boxes of $\xi$
in rows of the strip and columns of the strip, respectively. In particular, \eqref{eq:tableau} implies that in  one variable, the skew function $P_{\lambda/\mu}(a)$ vanishes unless the skew
shape $\lambda/\mu$ is a horizontal strip. We also use
\begin{equation}\label{eq:skew-Q}
 Q_{\lambda/\mu}=\frac{b_\lambda}{b_\mu}P_{\lambda/\mu}.
\end{equation}

\subsection{Nonnegative specializations}\label{sec:specializations}

We identify a Macdonald-nonnegative specialization $\rho:\Lambda\mapsto \mathbb R$ with its parameter triple
$(\alpha;\beta;\gamma)$ and write
$f(\alpha;\beta;\gamma)=f(\rho)$ for $f\in\Lambda$.
Here $\alpha=(\alpha_1\ge \alpha_2\ge\dots\ge 0)$ and $\beta=(\beta_1\ge \beta_2\ge \dots\ge 0)$ are summable and $\gamma\ge0$. In general, we \emph{do not} assume normalization and $ \sum_{i=1}^{\infty}\alpha_i+\sum_{j=1}^{\infty}\beta_j+\gamma$ can be any nonnegative real.
In this situation, the $\rho$--image of a power sum is given by the same formula
\eqref{eq_power_sums_specialization} for $r\ge2$, while
\begin{equation}\label{eq_specialization_mass}
 p_1(\rho)=\sum_i\alpha_i+\sum_j\beta_j+\gamma.
\end{equation}
Kerov-Matveev theorem \cite[Theorem 1.4]{Mat19} states that the triples $(\alpha;\beta;\gamma)$ above
exhaust the Macdonald-nonnegative specializations. These specializations also satisfy enhanced positivity involving skew functions, see \cite[Section 2.2.1]{BC14}: for any $\lambda$ and $\mu$,
\begin{equation}
\label{eq_enhanced_positivity}
 P_{\lambda/\mu}(\alpha;\beta;\gamma)\ge 0, \qquad Q_{\lambda/\mu}(\alpha;\beta;\gamma)\ge 0.
\end{equation}
If $(\widetilde\alpha;\widetilde\beta;\widetilde\gamma)$ denotes the
convention in \cite{BC14,BP15,Mat19}, then
\begin{equation}\label{eq:parameter-conversion}
 \alpha_i=\widetilde\alpha_i,\qquad
 \beta_j=\theta^{-1}\widetilde\beta_j,\qquad
 \gamma=\theta^{-1}\widetilde\gamma.
\end{equation}

A finite alphabet $x=(x_1,\ldots,x_d)$ gives the specialization
$(x_1,\ldots,x_d;0;0)$, and we retain the usual polynomial notation
$f(x)=f(x_1,\ldots,x_d)=f(x_1,\ldots,x_d;0;0)$. For $a\ge0$, scaling and union act on triples by
\begin{equation}\label{eq:specialization-operations}
 a\cdot(\alpha;\beta;\gamma)=(a\alpha;a\beta;a\gamma),\qquad
 (\alpha;\beta;\gamma)\sqcup(\widehat\alpha;\widehat\beta;\widehat\gamma)
 =(\alpha,\widehat\alpha;\, \beta,\widehat\beta;\, \gamma+\widehat\gamma),
\end{equation}
where after taking the union the parameters are rearranged in non-increasing order.
Thus $f(a\cdot\rho)=a^{\deg f}f(\rho)$ for homogeneous $f$.
The union satisfies $p_r(\sigma\sqcup\tau)=p_r(\sigma)+p_r(\tau)$.

\subsection{Coherence}\label{Section_coherence}

The Pieri coefficients $w_{q,t}(\mu,\lambda)$ in \eqref{eq:pieri}
define a natural coupling of the measures $M_n^\rho$ of \eqref{eq_coherent_measure} over all $n$.
Iterating the Pieri rule gives the path interpretation
\begin{equation}\label{eq:path-dimension}
 \ddim(\lambda)=
 \sum_{\varnothing=\lambda^{(0)}\nearrow\cdots\nearrow\lambda^{(n)}=\lambda}
 \prod_{r=1}^nw(\lambda^{(r-1)},\lambda^{(r)}).
\end{equation}
The chains over which the summation goes can be identified with standard Young tableaux of shape $\lambda$, i.e.\ fillings of the boxes of $\lambda$ with numbers $1,2,\dots,n$ with no repetitions and with entries increasing along the rows and columns: entry $m$ in a box means that this box is added at step $m$ in the chain. When $q=t$, all edge weights equal one, and we get the usual
branching graph for irreducible representations of symmetric groups, and \eqref{eq:path-dimension} identifies $\ddim$ with the total number of standard Young tableaux of shape $\lambda$.

For normalized $\rho$, assign each finite path
$\varnothing=\lambda^{(0)}\nearrow\cdots\nearrow\lambda^{(n)}$ the
cylinder probability
\begin{equation}\label{eq:cylinder}
 P_{\lambda^{(n)}}(\rho)
 \prod_{r=1}^nw(\lambda^{(r-1)},\lambda^{(r)}).
\end{equation}
The Pieri rule and $p_1(\rho)=1$ imply consistency as $n$ varies: the projection of \eqref{eq:cylinder} onto $\lambda^{(r)}$ has the law $M_r^\rho$ for each $1\le r \le n$.
These probabilities therefore define a measure $\Prob^\rho$ on infinite
paths. The marginal of the diagram at level $n$ is $M_n^\rho$ by \eqref{eq:path-dimension}.
This is the process of growing Young diagrams mentioned after Theorem~\ref{Theorem_uniform}.

\subsection{Transposition}
\label{Section_transposition}

The automorphism $\omega_{q,t}$ of $\Lambda$ is defined on the generators $p_r$ through
\begin{equation}\label{eq:omega}
 \omega_{q,t}(p_r)=(-1)^{r-1}\frac{1-q^r}{1-t^r}p_r
\end{equation}
has inverse $\omega_{t,q}$ and satisfies
\begin{equation}\label{eq:duality}
 \begin{split}
 \omega_{q,t}P_{\lambda/\mu}(\cdot;q,t)
 &=Q_{\lambda'/\mu'}(\cdot;t,q),\\
 \omega_{q,t}Q_{\lambda/\mu}(\cdot;q,t)
 &=P_{\lambda'/\mu'}(\cdot;t,q).
 \end{split}
\end{equation}
The normalization constants obey
$b_\lambda(q,t)b_{\lambda'}(t,q)=1$.

\begin{lemma}\label{lem:transposition}
Let $\rho=(\alpha;\beta;\gamma)$ be a normalized Macdonald-nonnegative
specialization at $(q,t)$, and set
\begin{equation}\label{eq:rho-dual}
 \rho^\vee=\theta^{-1}\cdot(\rho\circ\omega_{t,q}).
\end{equation}
Then $\rho^\vee=(\beta;\alpha;\gamma)$ is a normalized Macdonald-nonnegative
specialization at $(t,q)$, and
\begin{equation}\label{eq:transpose-measure}
 M_{n;q,t}^\rho(\lambda)=M_{n;t,q}^{\rho^\vee}(\lambda').
\end{equation}
\end{lemma}

\begin{proof}
Duality implies
\[
 P_\nu(\rho^\vee;t,q)
 =\theta^{-|\nu|}Q_{\nu'}(\rho;q,t)\ge0,
 \qquad p_1(\rho^\vee)=1.
\]
For $r\ge2$, direct substitution in \eqref{eq_power_sums_specialization} gives
\[
 p_r(\rho^\vee)=\sum_j\beta_j^r+
 (-1)^{r-1}\theta^{-r}\frac{1-t^r}{1-q^r}\sum_i\alpha_i^r.
\]
Since the value of $\theta$ for the pair $(t,q)$ is $\theta^{-1}$,
these power sums identify $\rho^\vee$ with $(\beta;\alpha;\gamma)$.

Apply $\omega_{q,t}$ to \eqref{eq_dimension_expansion} and use
\eqref{eq:duality}. Comparing coefficients of $P_{\lambda'}(\cdot;t,q)$
yields
\[
 \ddim(\lambda)b_{\lambda'}(t,q)=\theta^{-n}\tdim(\lambda').
\]
On the other hand,
\[
 P_{\lambda'}(\rho^\vee;t,q)
 =\theta^{-n}b_\lambda(q,t)P_\lambda(\rho;q,t).
\]
The two identities and $b_\lambda(q,t)b_{\lambda'}(t,q)=1$ give
\eqref{eq:transpose-measure}.
\end{proof}

\section{Binomial coupling of two specializations}\label{Section_binomial_coupling}

We explain how the measures \eqref{eq_coherent_measure} for two different specializations can be coupled together.

\begin{proposition}\label{Proposition_binomial}
Let $\rho=\sigma\sqcup\tau$, where $\sigma,\tau$ are Macdonald-nonnegative
specializations with $p_1(\sigma)=p$ and $p_1(\tau)=1-p$, $0\le p\le 1$.
Consider a distribution on pairs of Young diagrams $\mu\subseteq\lambda\in\Y_n$:
\begin{equation}\label{eq:binomial-joint}
 \Prob_n(\mu,\lambda)
 =\ddim(\lambda)P_\mu(\sigma)P_{\lambda/\mu}(\tau), \qquad \mu\subseteq\lambda, \qquad \lambda\in \Y_n.
\end{equation}
Then \eqref{eq:binomial-joint} is a probability distribution, whose $\lambda$-marginal is $M_n^\rho$, and $|\mu|\sim\Bin(n,p)$.
If $p>0$, then conditionally on $|\mu|=m$, $0\le m \le n$, we have
\begin{equation}\label{eq:binomial-conditional}
 \mu\sim M_m^{\bar\sigma},\qquad \bar\sigma=p^{-1}\cdot\sigma.
\end{equation}
\end{proposition}
\begin{proof}
 By \eqref{eq_enhanced_positivity}, the weight \eqref{eq:binomial-joint} is nonnegative. Summing over $\mu$, using \eqref{eq:branching} in which $X$--variables are specialized using $\sigma$ and $Y$--variables are specialized using $\tau$, we conclude that the $\lambda$--marginal of \eqref{eq:binomial-joint} is $M_n^\rho$. Since $M_n^\rho$ is a probability measure, this also implies that the weights \eqref{eq:binomial-joint} sum up to $1$.
 
 In order to understand the marginal distribution of $\mu$, we use the following identity in $\Lambda$: for any $\mu\in\Y_m$ and $n\ge m$ we have
 \begin{equation}\label{eq:skew-sum}
 \sum_{\lambda\in\Y_n}\ddim(\lambda)P_{\lambda/\mu}
 ={n\choose m}\ddim(\mu)p_1^{n-m}.
 \end{equation}
 This identity is obtained by expanding $p_1(X,Y)^n$ in two ways. On the one hand, \eqref{eq_dimension_expansion} and \eqref{eq:branching} give
 \begin{equation}
 \label{eq_x1}
 p_1(X,Y)^n=\sum_{\lambda\in\Y_n}\ddim(\lambda)P_\lambda(X,Y)=\sum_{\mu\subseteq\lambda\in\Y_n} \ddim(\lambda) P_{\lambda/\mu}(X) P_\mu(Y).
 \end{equation}
 On the other hand, using the binomial theorem and then \eqref{eq_dimension_expansion} gives
 \begin{equation}
 \label{eq_x2}
  p_1(X,Y)^n=\sum_{m=0}^n {n\choose m} p_1(Y)^m p_1(X)^{n-m}= \sum_{m=0}^n  {n\choose m} \sum_{\mu\in \Y_m} \ddim(\mu)P_\mu(Y) p_1(X)^{n-m}.
 \end{equation}
 Comparing the coefficient of $P_\mu(Y)$ in \eqref{eq_x1} and \eqref{eq_x2}, we get \eqref{eq:skew-sum}. Next, applying the specialization $\tau$ to \eqref{eq:skew-sum} and multiplying by $P_\mu(\sigma)$, we conclude that the marginal distribution of $\mu$ in \eqref{eq:binomial-joint} has the weight
 \begin{equation}
 \label{eq_x3}
  \mathrm{Prob}(\mu)= {n\choose m}  (1-p)^{n-m} \ddim(\mu) P_\mu(\sigma), \qquad \mu\in \Y_m, \qquad 0\le m \le n.
 \end{equation}
 Summing \eqref{eq_x3} over $\mu\in \Y_m$ with fixed $m$ using \eqref{eq_dimension_expansion} gives the desired binomial law for $m$:
 \begin{equation}
 \label{eq_x4}
  \mathrm{Prob}(m)= {n\choose m}  (1-p)^{n-m} p^m, \qquad 0\le m \le n.
 \end{equation}
 Dividing \eqref{eq_x3} by \eqref{eq_x4} we get the conditional distribution of $\mu\in\Y_m$.
\end{proof}

When working with binomial distributions, we use Hoeffding's inequality, which can be proven directly using Stirling's formula or from \cite[Theorem~1]{Hoe63}. If $X\sim\Bin(n,p)$,
$n\ge1$, and $0\le p\le1$, then, for every $\varepsilon>0$,
\begin{equation}\label{eq:binomial-tail}
 \mathrm{Prob} \bigl[X\ge(p+\varepsilon)n\bigr]\le e^{-2\varepsilon^2n},\qquad
 \mathrm{Prob} \bigl[X\le(p-\varepsilon)n\bigr]\le e^{-2\varepsilon^2n}.
\end{equation}

The coupling also gives a skew version of the expectation identity
\eqref{eq_expectation_identity}.

\begin{proposition}
Consider the joint distribution \eqref{eq:binomial-joint} with an arbitrary Macdonald-nonnegative specialization $\sigma$ and
$\tau=(y_1,\ldots,y_d;0;0)$, where $y_1\ge y_2\ge \dots \ge y_d>0$, $p=p_1(\sigma)$, and $p+\sum_{j=1}^dy_j=1$.
Then for every $x_1\ge \dots\ge x_d>0$, we have
\begin{equation}\label{eq:skew-tilted-expectation}
 \E_{\Prob_n}\left[
   \frac{P_{\lambda/\mu}(x_1,\dots,x_d)}{P_{\lambda/\mu}(y_1,\dots,y_d)}
 \right]
 =\left(p+\sum_{j=1}^dx_j\right)^n.
\end{equation}
\end{proposition}
\begin{proof}
After substituting the definition \eqref{eq:binomial-joint} and summing using \eqref{eq:branching} and then \eqref{eq_dimension_expansion}, the expectation equals
\begin{multline*}
 \sum_{\lambda\in\Y_n}\ddim(\lambda)
 \sum_{\mu\subseteq\lambda}P_\mu(\sigma)P_{\lambda/\mu}(x_1,\dots,x_d)
 =\sum_{\lambda\in\Y_n}\ddim(\lambda)P_\lambda(\sigma\sqcup(x_1,\dots,x_d;0;0))\\= p_1\bigl(\sigma\sqcup(x_1,\dots,x_d;0;0)\bigr)^n =\left(p+\sum_{j=1}^dx_j\right)^n. \qedhere
\end{multline*}
\end{proof}

\section{Specializations with no $\alpha$ parameters}\label{Section_noalpha}

The goal of this section is to show that if a specialization has all $\alpha_i$ parameters equal to zero, then the rows in the corresponding measures \eqref{eq_coherent_measure} grow sublinearly with $n$.

We begin with two uniform bounds on Macdonald weights \eqref{eq:b} and \eqref{eq:edge-formula}.
\begin{lemma}\label{lem:weights}
There is a constant $C_0=C_0(q,t)$ such that for any $\mu\nearrow\lambda$
\begin{equation}\label{eq:edge-bound}
 e^{-C_0}\le w_{q,t}(\mu,\lambda)\le e^{C_0}.
\end{equation}
There is also a constant $C_1=C_1(q,t)$ such that, for any Young diagram $\lambda$
and any set $D$ of boxes in one row of $\lambda$,
\begin{equation}\label{eq:row-b-bound}
 \sum_{s\in D}|\log b_\lambda(s)|\le C_1.
\end{equation}
\end{lemma}

\begin{proof}
Put $u=\max(q,t)\in[0,1)$. The elementary bound
$|\log(1-z)|\le z/(1-u)$, $0\le z\le u$, gives
\[
 |\log G(a,l)|\le\frac{4}{1-u}u^{a+l+1}.
\]
For the factors in \eqref{eq:edge-formula}, the integers
$(\mu_i-c)+(r-1-i)$, $i<r$, are distinct and nonnegative.
Their geometric sum is at most $u/(1-u)$, proving
\eqref{eq:edge-bound} with $C_0=4u/(1-u)^2$.

Similarly, \eqref{eq:b} gives
\[
 |\log b_\lambda(s)|\le\frac{2}{1-u}u^{a_\lambda(s)+1}.
\]
The arm lengths of boxes in one row are distinct and nonnegative.
Summing the geometric series proves \eqref{eq:row-b-bound}, for example
with $C_1=2/(1-u)^2$.
\end{proof}

We use \eqref{eq:edge-bound}, \eqref{eq:row-b-bound} to bound the first row for the Plancherel specialization $(0;0;1)$.

\begin{lemma}\label{lem:plancherel}
For every $0<\delta\le1$, as $n\to\infty$
\begin{equation}\label{eq:plancherel-tail}
 M_n^{(0;0;1)}\bigl[\lambda_1\ge\delta n\bigr]
 \le \exp\bigl(-\delta n\log n+O_{q,t,\delta}(n)\bigr).
\end{equation}
\end{lemma}

\begin{proof}
For the specialization $\rho=(0;0;1)$, its value $f(\rho)$ on a homogeneous symmetric function of degree $n$ is the coefficient of $p_1^n$ when $f$ is expanded as a linear combination of products of power sums. By \eqref{eq:scalar-product} and \eqref{eq:dimension-pairing}, the
coefficient of $p_1^n$ in the power sum expansion of $Q_\lambda$ is
$\theta^n\ddim(\lambda)/n!$. Hence, \eqref{eq_coherent_measure} becomes
\begin{equation}\label{eq:plancherel-law}
P_\lambda(0;0;1)=\frac{\theta^n\ddim(\lambda)}{n!b_\lambda}\qquad
\text{ and } \qquad
 M_n^{(0;0;1)}(\lambda)=
 \frac{\theta^n\ddim(\lambda)^2}{b_\lambda n!}.
\end{equation}
For $\ddim(\lambda)$ we use the formula \eqref{eq:path-dimension} representing it as a sum over standard Young tableaux.
Suppose $\lambda_1=r$, and write $\nu=(\lambda_2,\lambda_3,\ldots)$.
Deleting the first row of a standard Young tableau $T$ and standardizing the
remaining entries gives a standard tableau $\widehat T$ of shape $\nu$.
The tableau $T$ is uniquely determined by $\widehat T$ and the set of $r$
entries in the first row of $T$. By \eqref{eq:edge-bound}, the weight of $T$ in \eqref{eq:path-dimension} is
at most $e^{2C_0n}$ times that of $\widehat T$. Summing over tableaux gives
\[
 \ddim(\lambda)\le e^{2C_0n}\binom nr\ddim(\nu).
\]
Deleting the first row does not change the arm or leg lengths of the
remaining boxes. Thus only the factors corresponding to the first row of $\lambda$ remain in $b_\lambda/b_\nu$,
and \eqref{eq:row-b-bound} gives
\[
 \left|\log\frac{b_\lambda}{b_\nu}\right|\le C_1.
\]
Substituting these bounds into \eqref{eq:plancherel-law} yields
\begin{equation}\label{eq:plancherel-row-deletion}
 M_n^{(0;0;1)}(\lambda)
 \le e^{C_2n}\frac{\binom nr}{r!}\,M_{n-r}^{(0;0;1)}(\nu),
\end{equation}
where $C_2=C_2(q,t)$ and the factor $\theta^r$ is absorbed into the
exponential bound. We fix $r$ and sum \eqref{eq:plancherel-row-deletion} over
$\lambda$ with $\lambda_1=r$. Equivalently, this is summation of $M_{n-r}^{(0;0;1)}(\nu)$ over $\nu$ with $\nu_1\le r$ and we can use the fact that the total mass of the measure $M_{n-r}^{(0;0;1)}$ is $1$. Hence, for $m=\lceil\delta n\rceil$, bounding $\sum_r {n\choose r}\le 2^n$, we get
\[
 M_n^{(0;0;1)}\bigl[\lambda_1\ge m\bigr]
 \le e^{C_2n}\sum_{r=m}^n\frac{\binom nr}{r!}
 \le e^{C_2n}\frac{2^n}{m!}.
\]
Stirling's formula gives \eqref{eq:plancherel-tail}.
\end{proof}

Using Proposition \ref{Proposition_binomial} we add $\beta$ parameters to the statement of Lemma \ref{lem:plancherel}.

\begin{proposition}\label{Proposition_noalpha_bound}
Let $\rho=(0;\beta;\gamma)$ be a normalized Macdonald-nonnegative specialization. For every
$\delta>0$, there are $C,c>0$ such that
\begin{equation}\label{eq:noalpha}
 M_n^\rho\bigl[\lambda_1>\delta n\bigr]\le Ce^{-cn}.
\end{equation}
\end{proposition}

\begin{proof}
It suffices to consider $0<\delta<1$. Choose $L$ such that
\[
 r_L=\sum_{j>L}\beta_j<\delta/8.
\]
Apply Proposition~\ref{Proposition_binomial} with
\[
 \sigma=(0;\beta_1,\ldots,\beta_L;\gamma),\qquad
 \tau=(0;\beta_{L+1},\beta_{L+2},\ldots;0).
\]
It gives a coupling of Young diagrams $\mu\subseteq\lambda$ such that
\[
 R=|\lambda|-|\mu|\sim\Bin(n,r_L).
\]
Conditionally on its size, the law of $\mu$ is \eqref{eq_coherent_measure} associated with the specialization
$(1-r_L)^{-1}\cdot\sigma$.

Apply Proposition~\ref{Proposition_binomial} to the distribution of $\mu$ repeatedly $L$ times to split off the $L$
individual $\beta$ parameters, conditioning on the current size and
normalizing the remaining specialization at each step. This gives a chain of diagrams:
$$
 \mu^{(0)}\subseteq\mu^{(1)}\subseteq\dots \subseteq \mu^{(L)}=\mu\subseteq \lambda.
$$
When $\tau$ is a specialization in a single $\beta$-parameter, by \eqref{eq:duality}, the proof of Lemma \ref{lem:transposition}, and \eqref{eq:tableau}, the skew function $P_{\mu^{(r)}/\mu^{(r-1)}}(\tau)$ vanishes unless $\mu^{(r)}/\mu^{(r-1)}$
is a vertical strip. Therefore, at each step the first row changes by at most $1$. Hence, almost surely in this coupling,
\begin{equation}\label{eq:noalpha-coupling}
 \lambda_1\le\mu^{(L)}_1+R\le\mu^{(0)}_1+R+L.
\end{equation}
If $\gamma=0$, then $\mu^{(0)}=\varnothing$. Otherwise, conditionally on
$|\mu^{(0)}|=m$, its law is $M_m^{(0;0;1)}$. 

For $m<\delta n/2$, the event
$\{\mu^{(0)}_1>\delta n/2\}$ is impossible. For $m\ge\delta n/2$, there is an inclusion
$\{\mu^{(0)}_1>\delta n/2\}\subseteq \{\mu^{(0)}_1>\delta m/2\}$, because $m\le n$. Lemma~\ref{lem:plancherel}
therefore gives
\[
 \mathrm{Prob} \bigl[\mu^{(0)}_1>\delta n/2\bigr]\le Ce^{-cn}.
\]
For all sufficiently large $n$, one has $L\le\delta n/4$.
Then \eqref{eq:noalpha-coupling} implies
\[
 \{\lambda_1>\delta n\}
 \subseteq\{\mu^{(0)}_1>\delta n/2\}\cup\{R>\delta n/4\}.
\]
By the binomial bound \eqref{eq:binomial-tail} and $r_L<\delta/8$, the last event has
exponentially small probability.
\end{proof}

We also record a version of Proposition \ref{Proposition_noalpha_bound} for the inner diagram in Proposition \ref{Proposition_binomial}, whose size is random.

\begin{lemma}\label{lem:random-size}
Let $\sigma=(0;\beta;\gamma)$ be a Macdonald-nonnegative specialization with
$p=p_1(\sigma)$, $0<p<1$. Suppose $\mu$ is the inner diagram in Proposition \ref{Proposition_binomial}. For every
$\delta>0$,
\begin{equation}\label{eq:random-size}
 \Prob_n\bigl[\mu_1>\delta n\bigr]\le Ce^{-cn},
\end{equation}
where the constants $C,c>0$ may depend on $\sigma$ and $\delta$.
\end{lemma}

\begin{proof} Conditionally on $|\mu|=m$,
the law of $\mu$ is $M_m^{p^{-1}\cdot\sigma}$. The event in \eqref{eq:random-size}
is impossible for $m\le\delta n$. For $m>\delta n$, the event implies
$\mu_1>\delta m$, and Proposition \ref{Proposition_noalpha_bound} bounds its conditional
probability by $Ce^{-cm}\le Ce^{-c\delta n}$. Averaging over $m$ proves
the claim.
\end{proof}

\section{A monomial estimate for skew Macdonald polynomials}
\label{Section_localization_to_monomial}

In this section we approximate the values of skew Macdonald polynomials by their leading monomials. Let us fix $d\ge1$ and a skew diagram $\lambda/\mu$ with at most $d$ boxes in each column. We fill each column of $\lambda/\mu$ from top to bottom with $1,2,\ldots$,
starting at $1$ in every column, and denote the resulting filling by
$T_0$, as in Figure~\ref{fig:canonical-tableau}. Define its weight $\kappa=(\kappa_1,\ldots,\kappa_d)$ by
\begin{equation}\label{eq:kappa}
 \kappa_j=\#\{\text{boxes of }T_0\text{ with entry }j\},
 \qquad 1\le j\le d.
\end{equation}
Equivalently, $\kappa_j$ counts the skew columns containing at least $j$
boxes. In particular, $\kappa$ is a partition of $|\lambda|-|\mu|$.

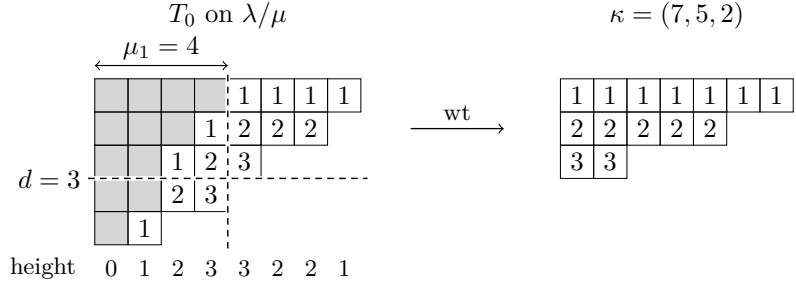
\begin{figure}[t]
\centering
\begin{tikzpicture}[x=0.44cm,y=0.44cm,font=\small,
  cell/.style={draw=black,line width=0.35pt},
  cut/.style={dash pattern=on 2.5pt off 2pt,line width=0.55pt,
    preaction={draw=white,line width=1.7pt}}]
  \node at (4,1.9) {$T_0$ on $\lambda/\mu$};
  \foreach \row/\length in {1/4,2/3,3/2,4/2,5/1} {
    \foreach \col in {1,...,\length} {
      \fill[black!16] (\col-1,-\row) rectangle (\col,1-\row);
    }
  }
  \foreach \row/\length in {1/8,2/7,3/5,4/4,5/2} {
    \foreach \col in {1,...,\length} {
      \draw[cell] (\col-1,-\row) rectangle (\col,1-\row);
    }
  }
  \foreach \col/\innerheight/\height in
    {1/5/0,2/4/1,3/2/2,4/1/3,5/0/3,6/0/2,7/0/2,8/0/1} {
    \ifnum\height>0
      \foreach \entry in {1,...,\height} {
        \node at (\col-0.5,-\innerheight-\entry+0.5) {$\entry$};
      }
    \fi
    \node[font=\footnotesize] at (\col-0.5,-5.7) {$\height$};
  }
  \node[anchor=east,font=\footnotesize] at (-0.25,-5.7) {height};
  \draw[<->,line width=0.45pt] (0,0.4) -- (4,0.4)
    node[midway,above,inner sep=2pt] {$\mu_1=4$};
  \draw[cut] (4,0.12) -- (4,-5.25);
  \draw[cut] (-0.2,-3) -- (8.2,-3);
  \node[anchor=east,inner sep=2pt] at (-0.25,-3) {$d=3$};

  \draw[->,line width=0.5pt] (9.5,-1.5) -- (12.3,-1.5)
    node[midway,above,inner sep=3pt,font=\footnotesize] {$\mathrm{wt}$};

  \begin{scope}[shift={(14,0)}]
    \node at (3.5,1.9) {$\kappa=(7,5,2)$};
    \foreach \entry/\length in {1/7,2/5,3/2} {
      \foreach \col in {1,...,\length} {
        \draw[cell] (\col-1,-\entry) rectangle (\col,1-\entry);
        \node at (\col-0.5,-\entry+0.5) {$\entry$};
      }
    }
  \end{scope}
\end{tikzpicture}
\caption{Left: tableau $T_0$ for $\lambda=(8,7,5,4,2)$,
$\mu=(4,3,2,2,1)$, and $d=3$. Right: $\kappa=(7,5,2)$ records the multiplicities of the labels $1,2,3$.}
\label{fig:canonical-tableau}
\end{figure}

\begin{lemma} \label{Lemma_S_kappa}
 The row lengths of $\kappa$ given by \eqref{eq:kappa} approximate those of $\lambda$:
\begin{equation}\label{eq:kappa-comparison}
 0\le\lambda_j-\kappa_j\le\mu_1,\qquad 1\le j\le d.
\end{equation}
In particular, the sums \eqref{eq:partial-sums} satisfy $0\le S_k(\lambda)-S_k(\kappa)\le k\mu_1$ for
$1\le k\le d$.
\end{lemma}
\begin{proof}
$\kappa_j$ counts the columns of $\lambda/\mu$ of height
at least $j$, whereas $\lambda_j$ counts the columns of $\lambda$
of height at least $j$. Since columns of $\lambda/\mu$ are subsets of columns of $\lambda$, every column counted by $\kappa_j$ is also counted
by $\lambda_j$, leading to $\lambda_j\ge \kappa_j$. To the right of the column $\mu_1$, the column heights for $\lambda$ and $\lambda/\mu$
coincide, so discrepancies can occur only among the first $\mu_1$
columns. Therefore, also $\lambda_j\le \kappa_j+\mu_1$. This proves \eqref{eq:kappa-comparison}. Summing it over
$j=1,\ldots,k$, we get $0\le S_k(\lambda)-S_k(\kappa)\le k\mu_1$. 
\end{proof}

\begin{theorem}\label{Theorem_monomial_localization}
For each $0\le q,t<1$, $d=1,2,\dots$, there exists $C=C(q,t,d)>0$ such that, for every skew
diagram $\lambda/\mu$ with at most $d$ boxes in each column and every
$x_1\ge\cdots\ge x_d>0$, we have
\begin{equation}\label{eq:localization}
 e^{-C(1+\mu_1)}x^\kappa
 \le P_{\lambda/\mu}(x_1,\ldots,x_d)
 \le e^{C(1+\mu_1)} (1+\lambda_1)^{d(d-1)}x^\kappa,
\end{equation}
where $\kappa$ is defined by \eqref{eq:kappa}.
\end{theorem}

\begin{proof}
We first notice that $T_0$ is a semistandard Young tableau of shape $\lambda/\mu$. Hence, it contributes a monomial
$x^{\kappa}$ to \eqref{eq:tableau}. For any other semistandard Young tableau $T$ in the sum \eqref{eq:tableau}, the $j$th box from the top of a
skew column has entry at least $j$, the entry of $T_0$ in that box. Since
$x_1\ge\cdots\ge x_d>0$, comparison box by box gives
\begin{equation}
\label{eq_x7}
 x^{\wt(T)}\le x^{\wt(T_0)}=x^\kappa.
\end{equation}
Hence $x^\kappa$ is a maximal monomial in the tableau expansion.

\smallskip

Let us now estimate the weights of these tableaux. For a pair of Young diagrams $\nu\subseteq\eta$, let $\hat R(\eta/\nu)$ denote the number of rows of $\eta$ which intersect $\eta/\nu$. By \eqref{eq:strip-weight} and \eqref{eq:row-b-bound},
\begin{equation}\label{eq_x5}
 |\log\psi_{\eta/\nu}|\le2C_1 \hat R(\eta/\nu).
\end{equation}
Take a semistandard Young tableau $T$ of shape $\lambda/\mu$ and represent it as a sequence $
 \mu=\lambda^{(0)}\subseteq\lambda^{(1)}\subseteq\cdots
 \subseteq\lambda^{(d)}=\lambda$, where each $\lambda^{(j)}/\lambda^{(j-1)}$ is a horizontal strip. We claim that:
\begin{equation}
\label{eq_x6}
  \sum_{j=1}^d \hat R(\lambda^{(j)}/\lambda^{(j-1)})
 \le d^2+d\mu_1.
\end{equation}
Indeed, the first $d$ rows contribute at most $d^2$ to the sum. The remaining rows have at most $d \mu_1$ boxes of $\lambda/\mu$ (cf.\ left panel in Figure \ref{fig:canonical-tableau}), and therefore contribute at most $d\mu_1$ to the sum. Combining \eqref{eq_x6} with \eqref{eq_x5}, we get a uniform bound for the weights in \eqref{eq:tableau}:
\begin{equation}\label{eq:tableau-weight-bound}
 |\log\psi_T|\le2C_1(d^2+d\mu_1).
\end{equation}
By keeping only one monomial corresponding to $T_0$ from the expansion \eqref{eq:tableau}, this implies the lower bound on $P_{\lambda/\mu}(x_1,\ldots,x_d)$ in \eqref{eq:localization}. 

For the upper bound, we need to additionally estimate the total number of semistandard Young tableaux of shape $\lambda/\mu$ with entries from $\{1,\dots,d\}$. Ignoring column constraints, a row of length at most $l$ has at most
$(l+1)^{d-1}$ weakly increasing fillings with $\{1,\dots,d\}$ (a filling is determined by multiplicities of $1,2,\dots,d$; each factor of the first $d-1$ multiplicities takes at most $l+1$ values and the last one is then determined).
The first $d$ rows therefore have at most $(1+\lambda_1)^{d(d-1)}$ fillings, where we added $1$ so that the expression still makes sense for the trivial case $\lambda_1=0$.
Below them there are at most $d\mu_1$ boxes of the skew diagram $\lambda/\mu$, which have at most
$d^{d\mu_1}$ assignments when all constraints are ignored. Thus
\begin{equation}\label{eq:tableau-count}
 \#\{T\}\le (1+\lambda_1)^{d(d-1)}d^{d\mu_1}.
\end{equation}
The upper bound in \eqref{eq:localization} now follows from
\eqref{eq:tableau}, the monomial bound \eqref{eq_x7}, \eqref{eq:tableau-weight-bound},
and \eqref{eq:tableau-count}. 
\end{proof}

\begin{corollary}\label{Corollary_tilt}
Fix $1\le k\le d$ and $s\in\R$. For $y_1\ge\cdots\ge y_d>0$, put
\[
 x_j=
 \begin{cases}
 e^sy_j,&j\le k,\\
 y_j,&j>k.
 \end{cases}
\]
There exists a constant $C=C(q,t,d)>0$, such that 
for every $\mu\subseteq\lambda$:
\begin{equation}\label{eq:tilt}
 P_{\lambda/\mu}(y_1,\dots,y_d)
 \le C (1+\lambda_1)^{d(d-1)}
       e^{\mu_1(C+k\max(s,0))-sS_k(\lambda)}P_{\lambda/\mu}(x_1,\dots,x_d),
\end{equation}
where $S_k(\lambda)=\lambda_1+\dots+\lambda_k$.
For $s\le0$, the sharper bound
\begin{equation}\label{eq:negative-tilt}
 P_{\lambda/\mu}(y_1,\dots,y_d)
 \le e^{-sS_k(\lambda)}P_{\lambda/\mu}(x_1,\dots,x_d)
\end{equation}
holds for every choice of nonnegative variables $y_1,\dots,y_d$, without an ordering assumption.
\end{corollary}

\begin{proof}
First suppose $s\le0$. For a tableau $T$ in \eqref{eq:tableau}, let
$m_k(T)$ be the number of entries in $\{1,\ldots,k\}$.
In each skew column, these entries occupy at most its first $k$ boxes.
Moving these boxes to the top of the corresponding column of $\lambda$
places them in the first $k$ rows. Thus $m_k(T)\le S_k(\lambda)$, and therefore
nonnegativity of the tableau weights gives
\[
 P_{\lambda/\mu}(x_1,\dots,x_d)
 =\sum_T\psi_T y^{\wt(T)}e^{s m_k(T)}
 \ge e^{sS_k(\lambda)}P_{\lambda/\mu}(y_1,\dots,y_d).
\]
This proves \eqref{eq:negative-tilt}, and hence also \eqref{eq:tilt} in the case $s\le 0$. 

Now suppose $s>0$. In this situation the $x_i$ inherit the ordering of the $y_i$, i.e.\ $x_1\ge x_2\ge\dots\ge x_d>0$.  If some column of the skew diagram $\lambda/\mu$ has height greater than $d$, then
both sides of \eqref{eq:tilt} vanish. Otherwise, we can apply \eqref{eq:localization} both to $P_{\lambda/\mu}(x_1,\dots,x_d)$ and $P_{\lambda/\mu}(y_1,\dots,y_d)$.
It remains to use Lemma \ref{Lemma_S_kappa}:
\[
 \frac{y^\kappa}{x^\kappa}
 =e^{-sS_k(\kappa)}
 \le e^{-sS_k(\lambda)+sk\mu_1}. \qedhere
\]
\end{proof}

\section{Exponential concentration and the law of large numbers}
\label{Section_concentration}

In this section we finish the proof of Theorem \ref{Theorem_uniform}. We argue in terms of partial sums:
\[
 A_k=\sum_{i=1}^k\alpha_i, \qquad  S_k(\lambda)=\sum_{i=1}^k\lambda_i, \qquad k\ge 1,\qquad   A_0=S_0(\lambda)=0.
\]
Our intermediate goal is to prove that, for every normalized $\rho$, any $k\ge1$, and
$\varepsilon>0$,
\begin{equation}\label{eq:row-concentration}
 M_n^\rho\left[\left|\frac{S_k(\lambda)}n-A_k\right|>\varepsilon\right]
 \le Ce^{-cn},
\end{equation}
with constants $C,c>0$ that may depend on $q,t,\rho,k,\varepsilon$.

\subsection{Finitely many $\alpha$ parameters}\label{sec:ordinary-example}
We first illustrate the argument when
\[
 \rho=(\alpha_1,\ldots,\alpha_d;0;0),\qquad
 \alpha_1\ge\cdots\ge\alpha_d>0,\qquad \sum_{j=1}^d\alpha_j=1.
\]
From \eqref{eq:tableau}, the support of $M^\rho_n$ consists of diagrams with at most $d$ rows. Fix $k<d$, choose $s\in\mathbb R$ and apply \eqref{eq_expectation_identity} (which is also \eqref{eq:skew-tilted-expectation} with $\mu=\emptyset$) for 
\begin{equation}
\label{eq_tilted_parameters}
 x_1=e^s \alpha_1,\,x_2=e^s \alpha_2,\,\dots,x_k=e^s \alpha_k,\,\, x_{k+1}=\alpha_{k+1},\, \dots, x_d=\alpha_d.
\end{equation}
Recalling that $\alpha_1+\dots+\alpha_d=1$, we get
\begin{equation}\label{eq:ordinary-tilted-expectation}
 \E_{M_n^\rho}\left[\frac{P_\lambda(x_1,\dots,x_d)}{P_\lambda(\alpha_1,\dots,\alpha_d)}\right]
 =\bigl(1+A_k(e^s-1)\bigr)^n.
\end{equation}
Corollary \ref{Corollary_tilt} with $\mu=\emptyset$, and the bound $\lambda_1\le n$, transforms
\eqref{eq:ordinary-tilted-expectation} into
\begin{equation*}
 \E_{M_n^\rho} \exp\left(sS_k(\lambda)\right)
 \le
 C(n+1)^{d(d-1)}\bigl(1+A_k(e^s-1)\bigr)^n.
\end{equation*}
Dividing both sides by $e^{sn A_k}$, we get
\begin{equation}\label{eq:ordinary-mgf_2}
 \E_{M_n^\rho}\exp\left(s(S_k(\lambda)-nA_k)\right)
 \le
 C(n+1)^{d(d-1)} \exp\left[n\left(\ln\bigl(1+A_k(e^s-1) \bigr) -sA_k\right)\right].
\end{equation}
The function $f(s):=\ln\bigl(1+A_k(e^s-1) \bigr) -sA_k$ satisfies $f(0)=0$ and $f'(0)=0$, i.e.\, $f(s)=O(s^2)$ for small $s$. We combine this observation with Markov's inequality for the expectation in \eqref{eq:ordinary-mgf_2}: for $\eps>0$
\begin{align*}
 \mathrm{Prob}\bigl[S_k(\lambda)-nA_k>\eps n\bigr]&\le  C(n+1)^{d(d-1)} \exp\left[n\left(f(s)-\eps s\right)\right], \quad s>0,\\
 \mathrm{Prob}\bigl[S_k(\lambda)-nA_k<-\eps n\bigr]&\le  C(n+1)^{d(d-1)} \exp\left[n\left(f(s)+\eps s\right)\right], \quad s<0.
\end{align*}
For small $s$ both expressions in the exponents are negative, and therefore both probabilities go to zero exponentially fast as $n\to\infty$. This proves \eqref{eq:row-concentration} for $1\le k \le d-1$. For $k\ge d$ the bound \eqref{eq:row-concentration} is trivial, because $S_k=n$ and $A_k=1$.

\subsection{Finitely many $\alpha$ parameters, arbitrary $\beta$ and $\gamma$ parameters} \label{Section_finite_alpha}
In this section we work with $\rho=(\alpha;\beta;\gamma)$, such that $\alpha_1\ge\alpha_2\ge\dots\ge \alpha_{d}>0$, $\alpha_{d+1}=\alpha_{d+2}=\dots=0$. There are no restrictions on the $\beta$ and $\gamma$ parts, except for the total normalization $\sum_i\alpha_i+\sum_j\beta_j+\gamma=1$. Our goal is to again prove \eqref{eq:row-concentration} by generalizing the arguments of the previous section.

If $d=0$, then the assertion is Proposition~\ref{Proposition_noalpha_bound}. If there are no $\beta$s or $\gamma$, then the previous section applies. Otherwise, we decompose the specialization
\[
 \rho=(\alpha_1,\dots,\alpha_d;\, \beta;\, \gamma)=\sigma\sqcup \tau, \qquad
 \sigma=(0;\, \beta;\, \gamma), \qquad \tau=(\alpha_1,\dots\alpha_d;\, 0;0).
\]
Proposition \ref{Proposition_binomial} gives the joint law
\begin{equation}\label{eq:finite-joint}
 \Prob_n(\mu,\lambda)
 =\ddim(\lambda)P_\mu(\sigma)P_{\lambda/\mu}(\alpha_1,\dots,\alpha_d).
\end{equation}
The inner diagram $\mu$ contains the contribution of the $\beta$ and $\gamma$ parameters. 
Lemma \ref{lem:random-size} gives, for every $\delta>0$,
\begin{equation}\label{eq:inner-small}
 \Prob_n \bigl[\mu_1>\delta n\bigr]\le C_\delta e^{-c_\delta n}.
\end{equation}
This bound will allow us to control the $\mu_1$ part of the bound \eqref{eq:tilt}.

We choose $1\le k\le d$, $s\in\mathbb R$, and introduce the tilted variables \eqref{eq_tilted_parameters}. The skew expectation identity \eqref{eq:skew-tilted-expectation} and $\sum_i\alpha_i+\sum_j\beta_j+\gamma=1$ give
\begin{equation}\label{eq:finite-tilted-expectation}
 \E_{\Prob_n}\left[\frac{P_{\lambda/\mu}(x_1,\dots,x_d)}{P_{\lambda/\mu}(\alpha_1,\dots,\alpha_d)}\right]=\bigl(1+A_k(e^s-1)\bigr)^n.
\end{equation}
Similarly to \eqref{eq:ordinary-mgf_2}, Corollary \ref{Corollary_tilt} transforms \eqref{eq:finite-tilted-expectation} into
\begin{equation}\label{eq:skew-mgf_2}
 \E_{\Prob_n}\exp\left(s(S_k(\lambda)-nA_k)-\mu_1 C_3\right)
 \le
 C(n+1)^{d(d-1)} \exp\left[n\left(\ln\bigl(1+A_k(e^s-1) \bigr) -sA_k\right)\right],
\end{equation}
where the constant $C_3$ stays bounded as long as $k$ and $s$ are bounded. Using $f(s)=\ln\bigl(1+A_k(e^s-1) \bigr) -sA_k$, we apply Markov's inequality to \eqref{eq:skew-mgf_2} and get the bounds:
\begin{align*}
 \mathrm{Prob}\bigl[S_k(\lambda)-nA_k>\eps n, \quad \mu_1\le \delta n\bigr]&\le  C(n+1)^{d(d-1)} \exp\left[n\left(f(s)-\eps s+\delta C_3\right)\right], \quad s>0,\\
 \mathrm{Prob}\bigl[S_k(\lambda)-nA_k<-\eps n, \quad \mu_1\le \delta n\bigr]&\le  C(n+1)^{d(d-1)} \exp\left[n\left(f(s)+\eps s+\delta C_3\right)\right], \quad s<0.
\end{align*}
For a given $\eps>0$, we first choose $s_1>0>s_2$, so that $f(s_1)-\eps s_1<0$ and $f(s_2)+\eps s_2<0$, and then choose $\delta$ small enough, so that $f(s_1)-\eps s_1+\delta C_3<0$ and $f(s_2)+\eps s_2+\delta C_3<0$. Hence, for the choices $s=s_1$ and $s=s_2$, the expressions in the exponents are negative, and we conclude that the probability
$$
  \mathrm{Prob}\bigl[ \bigl|S_k(\lambda)-nA_k\bigr|>\eps n, \quad \mu_1\le \delta n\bigr]
$$
goes to $0$ exponentially fast as $n\to\infty$. Combining with \eqref{eq:inner-small}, we get \eqref{eq:row-concentration}.

Now let $k>d$. By \eqref{eq:tableau}, on the support of \eqref{eq:finite-joint} every column
of $\lambda/\mu$ has height at most $d$. In particular,
$\lambda_{d+1}\le\mu_1$: otherwise column $\mu_1+1$ would have at least
$d+1$ boxes. Hence,
\[
 S_k(\lambda)\le S_d(\lambda)+(k-d)\mu_1.
\]
Therefore, using $A_k=A_d$ and \eqref{eq:inner-small}, the bound \eqref{eq:row-concentration} for $k>d$ follows from the $k=d$ case.

\subsection{Arbitrary $(\alpha;\beta;\gamma)$}
In this section we pass from finitely many non-zero $\alpha$ parameters to an arbitrary summable
sequence. The idea is to use the binomial decomposition of Proposition \ref{Proposition_binomial} to compare with the case of finitely many $\alpha$ parameters.

We fix an arbitrary normalized Macdonald-nonnegative specialization $\rho$ with parameters
\eqref{eq_parameter_space}. Take $0<\varepsilon<1$ and $k=1,2,\dots$. Choose $d\ge k$ such that
\[
 r=\sum_{j>d}\alpha_j<\varepsilon/8.
\]
If $r=0$, then we are back to the setting of the previous section. Otherwise, split the specialization $\rho$ as
\[
 \rho=\sigma\sqcup \tau,\qquad \sigma=(\alpha_1,\ldots,\alpha_d;\beta;\gamma),\qquad
 \tau=(\alpha_{d+1},\alpha_{d+2},\ldots;0;0).
\]
In the coupling $\mu\subseteq\lambda\in\Y_n$ of Proposition \ref{Proposition_binomial}, write
\[
 m=|\mu|,\qquad R=n-m.
\]
We have $R\sim\Bin(n,r)$ and, conditionally on $m$,
$\mu$ has distribution $M_m^{\bar\sigma}$, where $\bar\sigma=(1-r)^{-1}\cdot\sigma$.
The first $k$ $\alpha$-parameters of $\bar\sigma$ sum up to
$\bar A_k=A_k(1-r)^{-1}$. We have
\begin{equation}\label{eq_x8}
 A_k m \le\bar A_k m \le A_k n (1-r)^{-1}.
\end{equation}
In addition, since $\lambda$ differs from $\mu$ by adding $R$ boxes, we have
\begin{equation}\label{eq_x9}
 S_k(\lambda)-R\le S_k(\mu) \le S_k(\lambda).
\end{equation}
We next use the exponential concentration \eqref{eq:binomial-tail}. Since $r<\eps/8$, it implies that as $n\to\infty$ the inequality
\begin{equation}
\label{eq_x10}
 (1-\eps/4)n \le m \le n
\end{equation}
holds with probability exponentially close to $1$. Simultaneously, by Section \ref{Section_finite_alpha}, the inequality
\begin{equation}
\label{eq_x11}
 |S_k(\mu)-\bar A_k m|<\eps n /8
\end{equation}
also holds with probability exponentially close to $1$. Altogether, the inequalities \eqref{eq_x8}-\eqref{eq_x11} imply
$$
 |S_k(\lambda)- A_k n|<\eps n.
$$

\begin{remark}\label{rem:binomial-lower}
The lower-tail bound $S_k\ge n (A_k-\eps)$ (with probability exponentially close to $1$) can be deduced without the approximation of Macdonald polynomials of Theorem \ref{Theorem_monomial_localization}. Split off
$\sigma=(\alpha_1,\ldots,\alpha_k;0;0)$ from $\rho$. The resulting inner diagram
$\mu$ has at most $k$ rows, and therefore:
\[
 S_k(\lambda)\ge S_k(\mu)=|\mu|,\qquad
 |\mu|\sim\Bin(n,A_k).
\]
Hoeffding's inequality \eqref{eq:binomial-tail} applied to $|\mu|$ then gives the desired lower bound on $S_k(\lambda)$.
This argument does not yield the upper-tail bound, since it
does not control how many boxes the complementary specialization
contributes to the first $k$ rows. 
\end{remark}

\subsection{Completion of the proof}\label{sec:completion}
\begin{proof}[Proof of Theorem~\ref{Theorem_uniform}]
Applying \eqref{eq:row-concentration} with $\eps$ replaced by $\eps/2$ and noticing
\[
 \lambda_i=S_i(\lambda)-S_{i-1}(\lambda),\qquad
 \alpha_i=A_i-A_{i-1},
\]
we conclude that for each individual $k$ and with constants $c$ and $C$ depending on $k$,
$$
 M_n^\rho\left[\left|\frac{\lambda_k}n-\alpha_k\right|>\varepsilon\right]
 \le Ce^{-cn}, \qquad k=1,2,\dots.
$$
Applying this row estimate at parameters $(t,q)$ to
$\rho^\vee=(\beta;\alpha;\gamma)$, and using Lemma~\ref{lem:transposition}, we get
$$
 M_n^\rho\left[\left|\frac{\lambda'_k}n-\beta_k\right|>\varepsilon\right]
 \le Ce^{-cn}, \qquad k=1,2,\dots.
$$

To obtain uniformity over all coordinates, note that monotonicity and
normalization imply
\[
 0\le\frac{\lambda_i}{n},\alpha_i\le\frac1i,
 \qquad
 0\le\frac{\lambda'_j}{n},\beta_j\le\frac1j.
\]
Choose $m$ with $(m+1)^{-1}<\varepsilon$. Indices greater than $m$
cannot contribute to the event in \eqref{eq_uniform_concentration}.
A finite union bound over the remaining indices proves
\eqref{eq_uniform_concentration}.
\end{proof}

\end{document}